\documentclass{amsproc}%
\usepackage{graphicx}
\usepackage{amsfonts}%
\usepackage{amsmath}%
\usepackage{amssymb}
\providecommand{\U}[1]{\protect\rule{.1in}{.1in}}
\theoremstyle{plain}

\newtheorem{corollary}{Corollary}

\newtheorem{definition}{Definition}
\newtheorem{example}{Example}

\newtheorem{remark}{Remark}

\newtheorem{theorem}{Theorem}
\numberwithin{equation}{section}
\ifx\pdfoutput\relax\let\pdfoutput=\undefined\fi
\newcount\msipdfoutput
\ifx\pdfoutput\undefined\else
\ifcase\pdfoutput\else
\ifx\paperwidth\undefined\else
\ifdim\paperheight=0pt\relax\else\pdfpageheight\paperheight\fi
\ifdim\paperwidth=0pt\relax\else\pdfpagewidth\paperwidth\fi
\fi\fi\fi
\begin{document}
\title[Characterizations of Beta and Gamma functions]{Directional convexity and characterizations of Beta and Gamma functions}
\author{Martin Himmel}
\curraddr{Faculty of Mathematics, Computer Science and Econometrics University of
Zielona G\'{o}ra, Szafrana 4A, PL 65-516 Zielona G\'{o}ra, Poland}
\email{himmel@mathematik.uni-mainz.de}
\author{Janusz Matkowski}
\curraddr{Faculty of Mathematics, Computer Science and Econometrics University of
Zielona G\'{o}ra, Szafrana 4A, PL 65-516 Zielona G\'{o}ra, Poland}
\email{J.Matkowski@wmie.uz.zgora.pl }

\begin{abstract}
The logarithmic convexity of restrictions of the Beta functions to rays
parallel to the main diagonal and the functional equation
\[
\phi\left(  x+1\right)  =\frac{x\left(  x+k\right)  }{\left(  2x+k+1\right)
\left(  2x+k\right)  }\phi\left(  x\right)  ,\ \ \ \ \ \ x>0,
\]
for $k>0$ allow to get a characterizations of the Beta function. This fact and
a notion of the beta-type function lead to a new characterization of the Gamma function.

\end{abstract}
\maketitle

\section{\bigskip Introduction}

\footnotetext{\textit{2010 Mathematics Subject Classification. }Primary:
33B15, 26B25, 39B22.
\par
\textit{Keywords and phrases:} Gamma function, Beta function, beta-type
function, logarithmical convexity, geometrical convexity, directional
convexity, functional equation.}

The Euler Gamma function and its two variable companion, the Beta function,
both playing important role in applications, have found much attention in real
and complex analysis. The Gamma function, actually its restriction to the set
of positive reals, is logarithmically convex and, as an extension of the
$\left(  n-1\right)  !,$ satisfies the functional equation%
\[
f\left(  x+1\right)  =xf\left(  x\right)  ,\text{ \ \ \ \ }x>0;\text{
\ \ \ \ \ \ }f\left(  1\right)  =1.\text{\ }%
\]
Bohr and Mollerup \cite{Bohr}, 1922, proved that these two properties
characterize the Gamma function. The class of logarithmically convex functions
is rather narrow - for instance no power function has this property. Gronau
and Matkowski \cite{DG-JMI} improved Bohr-Mollerup theorem showing that the
Gamma function is the only solution of this equation that is geometrically
convex in a vicinity of $\ +\infty$. Applying this result, Alzer and Matkowski
\cite{Alzer} gave a characterization of the Gamma functions based on the
convexity with respect to the power means. \ 

It is known [Dragomir, Agarval, Barnett \cite{Drag}] that the Beta function is
logarithmically convex as a two variable function in $\left(  0,\infty\right)
^{2}$. It follows that the Beta function is \textit{logarithmically convex in
the direction of the main diagonal}, that is, for every $k>0$, the restriction
of the Beta function to the ray $\left\{  \left(  x,x+k\right)  :x>0\right\}
$ is logarithmically convex, and, moreover, solves the functional equation
\[
\phi\left(  x+1\right)  =\frac{x\left(  x+k\right)  }{\left(  2x+k+1\right)
\left(  2x+k\right)  }\phi\left(  x\right)  ,\ \ \ \ \ \ x>0.
\]
In section \ref{Main}, applying the Krull theorem \cite{Krull}, 1949, we prove
that this functional equation, the logarithmic convexity in the direction of
the main diagonal (much weaker that the logarithmic convexity), and a natural
initial condition characterize the Beta function (Theorem
\ref{BetaCharacterization}). Employing the method of Gronau-Matkowski
\cite{DG-JMII}, we show that the same functional equation and the geometrical
convexity also allow to get characterization of the Beta function (Theorem
\ref{GMBeta}). In section \ref{BetaTypeFunctions}, given a function
$\gamma:\left(  a,\infty\right)  \rightarrow\left(  0,\infty\right)  $,
($a\geq-\infty$), we define the two place function $B_{\gamma}:\left(
a,\infty\right)  ^{2}\rightarrow\left(  0,\infty\right)  ,$ referred to as
\textit{beta-type function of a generator} $\gamma$, by
\[
B_{\gamma}\left(  x,y\right)  :=\frac{\gamma\left(  x\right)  \gamma\left(
y\right)  }{\gamma\left(  x+y\right)  }\text{, \ \ \ \ \ \ \ \ \ }%
x,y>a\text{,}%
\]
and we observe that the equality $B_{\gamma_{2}}=B_{\gamma_{1}}$ holds iff
$\frac{\gamma_{2}}{\gamma_{1}}$ is an exponential function. It turns out that
this fact and the the above mentioned results about Beta function\ allow to
obtain new characterizations of the Gamma function. In the final section the
directional convexity is briefly discussed.

\section{Preliminaries and auxiliary results}

Recall that the Euler Gamma function $\Gamma:\left(  0,\infty\right)
\rightarrow\left(  0,\infty\right)  $, defined by%

\[
\Gamma\left(  x\right)  :=\int_{0}^{\infty}e^{-t}t^{x-1}%
dt,\ \ \ \ \ \ \ \ x>0,
\]
\ \ satisfies the functional equation%

\begin{equation}
f\left(  x+1\right)  =xf\left(  x\right)  ,\text{ \ \ \ \ }x>0;\text{
\ \ \ \ \ \ }f\left(  1\right)  =1.\text{\ }\tag{1}%
\end{equation}

The Beta function, $\mathcal{B}:\left(  0,\infty\right)  ^{2}\rightarrow
\left(  0,\infty\right)  ,$ defined by%

\begin{equation}
\mathcal{B}\left(  x,y\right)  :=\frac{\Gamma\left(  x\right)  \Gamma\left(
y\right)  }{\Gamma\left(  x+y\right)  },\ \ \ \ \ \ \ \ \ x,y>0,\tag{2}%
\end{equation}
playing important role in applications, is symmetric, that is
\[
\mathcal{B}\left(  x,y\right)  =\mathcal{B}\left(  y,x\right)
,\ \ \ \ \ \ \ \ \ x,y>0,
\]
and has the following integral representation%

\begin{equation}
\mathcal{B}\left(  x,y\right)  =\int_{0}^{1}t^{x-1}\left(  1-t\right)
^{y-1}dt,~\text{\ \ \ \ \ \ \ }x,y>0.~\tag{3}%
\end{equation}

Krull (\cite{Krull}, cf. also Kuczma \cite{Kuczma1}, pp. 114-118) proved the following

\begin{theorem}
\label{Krull}Let $a\geq-\infty$ be arbitrarily fixed. Suppose that $F:\left(
a,\infty\right)  \rightarrow\mathbb{R}$ is convex or concave, and%
\[
\lim_{x\rightarrow\infty}\left[  F\left(  x+1\right)  -F\left(  x\right)
\right]  =0.
\]
Then for every fixed $\left(  x_{0,},y_{0}\right)  \in\left(  a,\infty\right)
\times\mathbb{R}$ there exists exactly one convex or concave function
$\varphi:\left(  a,\infty\right)  \rightarrow\mathbb{R}$ satisfying the
functional equation%
\begin{equation}
\varphi\left(  x+1\right)  =\varphi\left(  x\right)  +F\left(  x\right)
,\text{ \ \ \ \ \ \ \ \ }x>a\tag{4}%
\end{equation}
such that%
\[
\varphi\left(  x_{0}\right)  =y_{0};
\]
moreover, for all $x>a$,%
\begin{equation}
\varphi\left(  x\right)  =y_{0}+\left(  x-x_{0}\right)  F\left(  x_{0}\right)
\tag{5}%
\end{equation}%
\[
-\sum_{n=0}^{\infty}\left\{  F\left(  x+n\right)  -F\left(  x_{0}+n\right)
-\left(  x-x_{0}\right)  \left[  F\left(  x_{0}+n+1\right)  -F\left(
x_{0}+n\right)  \right]  \right\}  .
\]

\end{theorem}

Since for any $c>a$\ and any function $\varphi_{0}:\left[  c,c+1\right]
\rightarrow\mathbb{R}$ satisfying the initial condition $\varphi_{0}\left(
c+1\right)  =\varphi_{0}\left(  c\right)  +F\left(  c\right)  $ there exists a
unique solution $\varphi:\left(  a,\infty\right)  \rightarrow\mathbb{R}$ of
the functional equation (4) such that $\varphi|_{\left[  c,c+1\right]
}=\varphi_{0}$ (cf. Kuczma \cite{Kuczma1}, p. 70), the following meaningful
improvement of Krull's theorem holds true.

\begin{theorem}
\label{KrulIimproved}Let $a\geq-\infty$ be arbitrarily fixed. Suppose that
$F:\left(  a,\infty\right)  \rightarrow\mathbb{R}$ is convex (concave), and
\[
\lim_{x\rightarrow\infty}\left[  F\left(  x+1\right)  -F\left(  x\right)
\right]  =0.
\]

Then, for every fixed $\left(  x_{0,},y_{0}\right)  \in\left(  a,\infty
\right)  \times\mathbb{R}$, there exists exactly one function $\varphi:\left(
a,\infty\right)  \rightarrow\mathbb{R}$ satisfying the functional equation%
\[
\varphi\left(  x+1\right)  =\varphi\left(  x\right)  +F\left(  x\right)
,\text{ \ \ \ \ \ \ \ \ }x>a
\]
that is concave (convex) in some interval $\left(  b,\infty\right)  ,$ $b\geq
a$, and such that%
\[
\varphi\left(  x_{0}\right)  =y_{0};
\]
moreover, (5) holds for all $x>a$.
\end{theorem}

\bigskip

Let $I\subset\left(  0,\infty\right)  $ be an interval. A function
$\phi:I\rightarrow\left(  0,\infty\right)  $ is called: \textit{geometrically
convex,} if%
\[
\phi\left(  x^{t}y^{1-t}\right)  \leq\left[  \phi\left(  x\right)  \right]
^{t}\left[  \phi\left(  y\right)  \right]  ^{1-t}\text{, \ \ \ \ \ \ \ \ }%
x,y\in I\,;\text{ }t\in\left(  0,1\right)  ;
\]
\textit{geometrically concave,} if the reversed inequality is satisfied; and
\textit{geometrically affine,} if the equality holds.

It is easy to check that a function $\phi:\left(  0,\infty\right)
\rightarrow\left(  0,\infty\right)  $ is geometrically affine\ iff
$\phi\left(  x\right)  =ax^{p}$ for some $a>0,$ and $p\in\mathbb{R}$.

\begin{remark}
If $\phi:I\rightarrow\left(  0,\infty\right)  $ is Jensen geometrically
convex, that is
\[
\phi\left(  \sqrt{xy}\right)  \leq\sqrt{\phi\left(  x\right)  \phi\left(
y\right)  }\text{, \ \ \ \ \ \ \ \ }x,y\in I\,,
\]
and $\phi$ satisfies one of the following conditions: $\phi$ is continuous at
a point; or $\phi$ is bounded from above in a neighborhood of a point, or
$\phi$ is Lebesgue measurable, then $\phi$ is geometrically convex.

A function $\phi:I\rightarrow\left(  0,\infty\right)  $ is geometrically
convex (Jensen geometrically convex) if, and only if, the function
log$\circ\phi\circ\exp$ is convex (Jensen convex) in the interval $\log\left(
I\right)  .$
\end{remark}

From Gronau and Matkowski \cite{DG-JMII}, (Theorems 2 and 3), we have the following

\begin{theorem}
\label{GronauMatkowski}Let $G:\left(  0,\infty\right)  \rightarrow\left(
0,\infty\right)  $ be such that $\log\circ G$ is concave on some vicinity of
$\infty,$ and
\[
\lim_{x\rightarrow\infty}\frac{G\left(  x+\delta\right)  }{G\left(  x\right)
}=1\ \ \ \text{for\ some \ \ }\delta>0.
\]
Then, for every $c>0,$ there exists exactly one solution $\phi:\left(
0,\infty\right)  \rightarrow\left(  0,\infty\right)  $ of the functional
equation%
\[
\phi\left(  x+1\right)  =G\left(  x\right)  \phi\left(  x\right)  ,\text{
\ \ \ \ \ \ }x>0\text{,}%
\]
such that $\phi$ is geometrically convex on some vicinity of $\infty$ and
$\phi\left(  1\right)  =c.$ Moreover, for all $x>0,$%
\[
\phi\left(  x\right)  =c\lim_{n\rightarrow\infty}G\left(  n\right)
^{-\frac{\log\left(  n+1+x\right)  -\log\left(  n+1\right)  }{\log\left(
n+1\right)  -\log n}}\frac{1}{G\left(  x\right)  }%
{\displaystyle\prod\limits_{j=1}^{n}}
\frac{G\left(  j\right)  }{G\left(  j+x\right)  }.
\]

\end{theorem}

\bigskip

\section{Characterization of the Beta function \label{Main}}

\bigskip

The main result of this section is the following

\begin{theorem}
\label{BetaCharacterization}Let $\mathcal{B}$ denote the Beta function and let
$k\geq0$ be arbitrarily fixed. If $\phi:\left(  0,\infty\right)
\rightarrow\left(  0,\infty\right)  $ is such that%
\begin{equation}
\phi\left(  x+1\right)  =\frac{x\left(  x+k\right)  }{\left(  2x+k+1\right)
\left(  2x+k\right)  }\phi\left(  x\right)
,\ \ \ \ \ \ \ \ \ \ \ \ x>0;\tag{6}%
\end{equation}%
\begin{equation}
\phi\left(  1\right)  =\mathcal{B}\left(  1,1+k\right)  ,\tag{7}%
\end{equation}
and$\ \log\circ\phi$ is convex, then
\begin{equation}
\phi\left(  x\right)  =\mathcal{B}\left(  x,x+k\right)
,\ \ \ \ \ \ \ \ \ \ \ x>0.\tag{8}%
\end{equation}

\end{theorem}

\begin{proof}
Let us assume that $\phi:\left(  0,\infty\right)  \rightarrow\left(
0,\infty\right)  $ satisfies equation (6), condition (7), put
\[
\varphi=\log\circ\phi,
\]
and define the function $F:\left(  0,\infty\right)  \rightarrow\left(
0,\infty\right)  $ by the formula%
\begin{equation}
F\left(  x\right)  =\log\left(  \frac{x\left(  x+k\right)  }{\left(
2x+k+1\right)  \left(  2x+k\right)  }\right)
,\ \ \ \ \ \ \ \ \ \ \ \ x>0.\tag{9}%
\end{equation}
Hence, taking the logarithm of both sides of equation (6), we get
\begin{equation}
\varphi\left(  x+1\right)  =\varphi\left(  x\right)  +F\left(  x\right)
,\text{ \ \ \ \ \ \ \ \ \ }x>0\text{,}\tag{10}%
\end{equation}
and, from (7),%
\begin{equation}
\varphi\left(  1\right)  =\log\mathcal{B}\left(  1,1+k\right)  .\tag{11}%
\end{equation}

Calculating the first and the second derivatives of $F$ we get%
\[
F^{\prime}\left(  x\right)  =\frac{1}{x}+\frac{1}{x+k}-\frac{2}{2x+k}-\frac
{2}{2x+k+1},\text{ \ \ \ \ \ \ }x>0,
\]
and, for every $x>0,$%
\begin{align*}
F^{\prime\prime}\left(  x\right)   &  =-\frac{1}{x^{2}}-\frac{1}{\left(
x+k\right)  ^{2}}+\frac{4}{\left(  2x+k\right)  ^{2}}+\frac{4}{\left(
2x+k+1\right)  ^{2}}\\
&  =-\frac{P\left(  x\right)  }{x^{2}\left(  x+k\right)  ^{2}\left(
2x+k\right)  ^{2}\left(  2x+k+1\right)  ^{2}},
\end{align*}
where%
\begin{align*}
P\left(  x\right)   & =16x^{5}+4x^{4}\left(  6k^{2}+10k+1\right)
+8kx^{3}\left(  k+1\right)  \left(  6k+1\right) \\
& +2k^{2}x^{2}\left(  k+1\right)  \left(  17k+5\right)  +2k^{3}x\left(
k+1\right)  \left(  5k+3\right)  +k^{4}\left(  k+1\right)  ^{2}.
\end{align*}
Since $k$ together with all occurring coefficients in the numerator are
non-negative, the second derivative $F^{\prime\prime}$ is non-positive and,
consequently, $F$ is concave on $\left(  0,\infty\right)  $. Moreover, from
the definition of $F$, we get%
\[
\lim_{x\rightarrow\infty}{\left[  F\left(  x+1\right)  -F\left(  x\right)
\right]  }%
\]%
\[
=\lim_{x\rightarrow\infty}{{\left[  \log\frac{\left(  x+1\right)  \left(
x+k+1\right)  }{\left(  2x+k+2\right)  \left(  2x+k+3\right)  }-\log
{\frac{x\left(  x+k\right)  }{\left(  2x+k+1\right)  \left(  2x+k\right)  }%
}\right]  }}%
\]%
\[
=\lim_{x\rightarrow\infty}{\log\frac{\left(  x+1\right)  \left(  x+k+1\right)
}{\left(  2x+k+2\right)  \left(  2x+k+3\right)  }\frac{\left(  2x+k\right)
\left(  2x+k+1\right)  }{x\left(  x+k\right)  }}%
\]%
\[
=\log1=0.
\]

Thus the function $F$ in equation (10) satisfies the suitable conditions of
Theorem \ref{Krull}. Taking into account that the function $\varphi$ is convex
solution of equation (10), applying the Theorem \ref{Krull} with
\[
a=0,\text{ \ \ \ }x_{0}=1,\text{ \ \ \ }y_{0}=\log\mathcal{B}\left(
1,1+k\right)  ,
\]
we conclude that $\varphi$ is the only function satisfying equation (10) such
that $\varphi\left(  1\right)  =\log\mathcal{B}\left(  1,1+k\right)  $, i.e.
satisfying condition (11).

To finish the proof, it is enough to show that, for a fixed $k\geq0$, the
function $f:\left(  0,\infty\right)  \rightarrow\mathbb{R}$ defined by%
\begin{equation}
f\left(  x\right)  :=\log\mathcal{B}\left(  x,x+k\right)  ,\text{
\ \ \ \ \ \ \ \ }x>0,\tag{12}%
\end{equation}
is convex and satisfies equation (10).

Though the convexity of $f$ follows from the convexity of $\log\circ
\mathcal{B}$ (\cite{Drag}), for the completeness, we present short argument.
Note that, by (3),
\[
\mathcal{B}(x,x+k)=\int_{0}^{1}{t^{x-1}\left(  1-t\right)  ^{x+k-1}dt},\quad
x>0.
\]
Hence, making use of (12), we have
\[
f^{\prime}{\left(  x\right)  }=\frac{\int_{0}^{1}{t^{x-1}\left(  1-t\right)
^{x+k-1}{\left(  \log{t}+\log{\left(  1-t\right)  }\right)  }}dt}%
{\mathcal{B}(x,x+k)},\quad x>0.
\]
and, for all $x>0,$%
\[
f^{\prime\prime}(x)\left[  \mathcal{B}\left(  x,x+k\right)  \right]
^{2}=\left(  \int_{0}^{1}\left[  g_{x}\left(  t\right)  \right]
^{2}dt\right)  \left(  \int_{0}^{1}\left[  h_{x}\left(  t\right)  \right]
^{2}dt\right)  -\left(  \int_{0}^{1}g_{x}\left(  t\right)  h_{x}\left(
t\right)  dt\right)  ^{2},
\]
where the functions $g_{x},h_{x}:\left(  0,1\right)  \rightarrow\mathbb{R}$
are defined by
\[
g_{x}\left(  t\right)  :=\left(  t^{x-1}\left(  1-t\right)  ^{x+k-1}\right)
^{\frac{1}{2}}\log\left(  t\left(  1-t\right)  \right)  \text{, \ \ \ \ \ \ }%
h_{x}\left(  t\right)  :=\left(  t^{x-1}\left(  1-t\right)  ^{x+k-1}\right)
^{\frac{1}{2}}.
\]

Since, by the Cauchy-Schwarz inequality,
\[
\left(  \int_{0}^{1}g_{x}\left(  t\right)  h_{x}\left(  t\right)  dt\right)
^{2}\leq\left(  \int_{0}^{1}\left[  g_{x}\left(  t\right)  \right]
^{2}dt\right)  \left(  \int_{0}^{1}\left[  h_{x}\left(  t\right)  \right]
^{2}dt\right)  ,
\]
we hence conclude that $f$ is convex.

From (2), taking into account that%
\[
\Gamma\left(  x+1\right)  =x\Gamma\left(  x\right)  \text{, \ \ \ \ \ \ }x>0,
\]
we have, for all $x,y>0,$
\begin{align*}
\mathcal{B}\left(  x+1,y+1\right)   & =\frac{\Gamma\left(  x+1\right)
\Gamma\left(  y+1\right)  }{\Gamma\left(  \left(  x+y+1\right)  +1\right)
}=\frac{xy\Gamma\left(  x\right)  \Gamma\left(  y\right)  }{\left(
x+y+1\right)  \Gamma\left(  x+y+1\right)  }\\
& =\frac{xy\Gamma\left(  x\right)  \Gamma\left(  y\right)  }{\left(
x+y+1\right)  \left(  x+y\right)  \Gamma\left(  x+y\right)  }=\frac
{xy}{\left(  x+y+1\right)  \left(  x+y\right)  }\mathcal{B}\left(  x,y\right)
.
\end{align*}
Setting here $y=x+k$ gives%
\[
\mathcal{B}\left(  x+1,x+k+1\right)  =\frac{x\left(  x+k\right)  }{\left(
2x+k+1\right)  \left(  2x+k\right)  }\mathcal{B}\left(  x,x+k\right)  ,\text{
\ \ \ \ \ \ }x>0,
\]
whence, by (12), the function $\varphi:=$ $f\ $\ satisfies the
functional\ equation%
\[
\varphi\left(  x+1\right)  =\varphi\left(  x\right)  +F\left(  x\right)
\text{, \ \ \ \ \ \ \ \ \ \ }x>0,
\]
where $F$ is defined by (9). Thus the function $\varphi=$ $f$ satisfies
equation (10) and, clearly, condition (11).

This completes the proof.
\end{proof}

\begin{remark}
Since, by (1) and (2), we have%
\[
\mathcal{B}\left(  1,1+k\right)  =\frac{\Gamma\left(  1\right)  \Gamma\left(
1+k\right)  }{\Gamma\left(  1+\left(  1+k\right)  \right)  }=\frac
{k\Gamma\left(  k\right)  }{\left(  1+k\right)  \Gamma\left(  k+1\right)
}=\frac{k\Gamma\left(  k\right)  }{\left(  1+k\right)  k\Gamma\left(
k\right)  }=\frac{1}{1+k},
\]
condition (7) is equivalent to%
\[
\phi\left(  1\right)  =\frac{1}{k+1}.
\]

\end{remark}

\begin{remark}
Since the Beta function is symmetric, the above theorem remains true on
replacing the functions $\left(  0,\infty\right)  \ni x\longmapsto\log
\circ\mathcal{B}\left(  x,x+k\right)  $ by $\left(  0,\infty\right)  \ni
x\longmapsto\log\circ\mathcal{B}\left(  x+k,x\right)  .$
\end{remark}

\bigskip

From Theorem \ref{KrulIimproved} and Theorem \ref{BetaCharacterization}\ we
obtain the following

\begin{corollary}
Let $\mathcal{B}$ denote the Beta function and let $k\geq0$ be arbitrarily
fixed. If $\phi:\left(  0,\infty\right)  \rightarrow\left(  0,\infty\right)  $
satisfies equation (6), condition (7), and $\phi$ is logarithmically convex in
an interval $\left(  a,\infty\right)  $ for some $a>0,$ then\ (8) holds true.
\end{corollary}

This corollary can be substantially strengthened with the aid of the
geometrical convexity:

\begin{theorem}
\label{GMBeta}Let $\mathcal{B}$ denote the Beta function and let $k\geq0$ be
arbitrarily fixed. If $\phi:\left(  0,\infty\right)  \rightarrow\left(
0,\infty\right)  $ satisfies equation (6), condition (7), and $\phi$ is
geometrically convex in an interval $\left(  a,\infty\right)  $ for some
$a>0$, then (8) holds true. Moreover, for all $x>0,$
\end{theorem}%

\begin{align*}
\phi\left(  x\right)   & =c\lim_{n\rightarrow\infty}\left(  \frac{n\left(
n+k\right)  }{\left(  2n+k+1\right)  \left(  2n+k\right)  }\right)
^{\frac{\log\left(  n+1+x\right)  \text{-log}\left(  n+1\right)  }{\log\left(
n+1\right)  -\log n}}\frac{\left(  2x+k+1\right)  \left(  2x+k\right)
}{x\left(  x+k\right)  }\\
& \prod\limits_{j=1}^{n}\frac{j\left(  j+k\right)  \left(  2\left(
j+x\right)  +k+1\right)  \left(  2\left(  j+x\right)  +1\right)  }{\left(
j+x\right)  \left(  j+x+k\right)  \left(  2j+k+1\right)  \left(  2j+1\right)
}.
\end{align*}

\begin{proof}
Define $G:\left(  0,\infty\right)  \rightarrow\left(  0,\infty\right)  ~$ by%
\[
G\left(  x\right)  =\frac{x\left(  x+k\right)  }{\left(  2x+k+1\right)
\left(  2x+k\right)  }.
\]

In the proof of Theorem \ref{BetaCharacterization} we have shown that the
concavity of the function $F:=\log\circ G$, and that
\[
\lim_{x\rightarrow\infty}\left[  \log G\left(  x+1\right)  -\log G\left(
x\right)  \right]  =0,
\]
which implies
\[
\lim_{x\rightarrow\infty}\frac{G\left(  x+\delta\right)  }{G\left(  x\right)
}=1
\]
with $\delta=1$. Thus all the assumptions of Theorem \ref{GronauMatkowski} are
satisfied. Consequently, for every $c>0,$ there exists exactly one solution
$\phi:\left(  0,\infty\right)  \rightarrow\left(  0,\infty\right)  $ of the
functional equation%
\[
\phi\left(  x+1\right)  =G\left(  x\right)  \phi\left(  x\right)
,\ \ \ \ \ x>0,
\]
such that $\phi$ is geometrically convex on $\left(  a,\infty\right)  $ for
some $a\geq0$ and $\phi\left(  1\right)  =c.$ The only thing left that has to
be verified is the validity of
\[
\phi\left(  x\right)  =\mathcal{B}\left(  x,x+k\right)  ,\text{ \ \ \ \ \ }%
x>0.
\]
Since $x\longmapsto\mathcal{B}\left(  x,x+k\right)  $ is an increasing
function and, due to Theorem \ref{BetaCharacterization}, logarithmically
convex, the mapping $x\longmapsto\mathcal{B}\left(  x,x+k\right)  $ is also
geometrically convex, see Lemma 1 in \cite{DG-JMI} Now clearly, putting
$c=\mathcal{B}\left(  1,1+k\right)  $ yields the claim.
\end{proof}

\section{Beta-type functions, their equality and characterization of the Gamma
function\label{BetaTypeFunctions}}

Taking into account the specific form (2) of\ the Beta function, we introduce
a broader class two variable \textit{beta-type functions} and we examine their properties.

\begin{definition}
\label{BetaDefinition}Let $a\geq-\infty$ be arbitrary fixed, and let $\gamma:$
$\left(  a,\infty\right)  \rightarrow\left(  0,\infty\right)  .$ The function
$B_{\gamma}:\left(  a,\infty\right)  ^{2}\rightarrow\left(  0,\infty\right)
$,%
\[
B_{\gamma}\left(  x,y\right)  :=\frac{\gamma\left(  x\right)  \gamma\left(
y\right)  }{\gamma\left(  x+y\right)  }\text{, \ \ \ \ \ \ \ \ }x,y>a\text{,}%
\]
is called \textit{beta-type function}, and the function $\gamma$ is said to be
its generator.
\end{definition}

\bigskip

\begin{remark}
Taking here $a=0$ and $\gamma=\Gamma$ where $\Gamma$\ is the Euler Gamma
function we obtain%
\[
\mathcal{B=}B_{\Gamma}.
\]
Thus the Beta function, denoted by $\mathcal{B}$, is a beta-type function, and
the Euler Gamma function is its generator.
\end{remark}

\begin{theorem}
Let $a\geq-\infty$ be fixed, and let $\gamma_{1},\gamma_{2}:$ $\left(
a,\infty\right)  \rightarrow\left(  0,\infty\right)  $ be arbitrary functions.

(i) The equality
\[
B_{\gamma_{1}}=B_{\gamma_{2}}%
\]
holds if, and only if, there exists a (unique) exponential function
\textsc{e}$:\mathbb{R\rightarrow}\left(  0,\infty\right)  ,$ i.e. satisfying
the functional equation%
\[
\text{\textsc{e}}\left(  x+y\right)  =\text{\textsc{e}}\left(  x\right)
\text{\textsc{e}}\left(  y\right)  \text{, \ \ \ \ \ \ }x,y\in\mathbb{R}%
\text{,}%
\]
and such that
\[
\gamma_{2}\left(  x\right)  =\text{\textsc{e}}\left(  x\right)  \gamma
_{1}\left(  x\right)  \text{, \ \ \ \ \ \ }x>a;
\]

(ii) Assume that the function $\frac{\gamma_{2}}{\gamma_{1}}$ is Lebesque
measurable, or continuous at a point, or its graph is not dense in $\left(
a,\infty\right)  \times\left(  0,\infty\right)  .$ Then $B_{\gamma_{1}%
}=B_{\gamma_{2}}$ if, and only if, there is (a unique) $c\in\mathbb{R}$ such
that%
\[
\gamma_{2}\left(  x\right)  =e^{cx}\gamma_{1}\left(  x\right)  \text{,
\ \ \ \ \ \ }x>a\text{.}%
\]

\end{theorem}

\begin{proof}
Assume that $B_{\gamma_{1}}=B_{\gamma_{2}}$ i.e., by Definition
\ref{BetaDefinition}, we have
\[
\frac{\gamma_{2}\left(  x\right)  \gamma_{2}\left(  y\right)  }{\gamma
_{2}\left(  x+y\right)  }=\frac{\gamma_{1}\left(  x\right)  \gamma_{1}\left(
y\right)  }{\gamma_{1}\left(  x+y\right)  }\text{, \ \ \ \ \ \ \ \ }x,y>a,
\]
or, equivalently,%
\[
\frac{\gamma_{2}\left(  x+y\right)  }{\gamma_{1}\left(  x+y\right)  }%
=\frac{\gamma_{2}\left(  x\right)  }{\gamma_{1}\left(  x\right)  }\frac
{\gamma_{2}\left(  y\right)  }{\gamma_{1}\left(  y\right)  }\text{,
\ \ \ \ \ \ \ \ }x,y>a,
\]
that is
\[
\frac{\gamma_{2}}{\gamma_{1}}\left(  x+y\right)  =\frac{\gamma_{2}}{\gamma
_{1}}\left(  x\right)  \frac{\gamma_{2}}{\gamma_{1}}\left(  y\right)  \text{,
\ \ \ \ \ \ \ \ }x,y>a.
\]
Clearly, the function $\frac{\gamma_{2}}{\gamma_{1}}$ can be extended to a
unique exponential function \textsc{e}$:\mathbb{R\rightarrow}\left(
0,\mathbb{\infty}\right)  $, which proves (i).

Part (ii) is a consequence of the suitable classical results (Acz\'{e}l
\cite{JA} Chapter II, Kuczma \cite{Kuczma2} Chapter XIII).
\end{proof}

Hence, taking into account the last remark, we obtain

\begin{corollary}
Let $\gamma:$ $\left(  0,\infty\right)  \rightarrow\left(  0,\infty\right)  $
be an arbitrary function. Then

(i) the equality $B_{\gamma}=\mathcal{B}$ holds if, and only if, there is
exponential function \textsc{e}$:\mathbb{R\rightarrow}\left(  0,\infty\right)
$ such that
\[
\gamma\left(  x\right)  =\text{\textsc{e}}\left(  x\right)  \Gamma\left(
x\right)  \text{, \ \ \ \ \ \ }x>0.
\]
(ii) assume that $\gamma$ is Lebesque measurable, or continuous at a point, or
its graph is not dense in $\left(  0,\infty\right)  ^{2};$ then $B_{\gamma
}=\mathcal{B}$ if, and only if, there is $c\in\mathbb{R}$ such that%
\[
\gamma\left(  x\right)  =e^{cx}\Gamma\left(  x\right)  \text{, \ \ \ \ \ \ }%
x>0\text{.}%
\]

\end{corollary}

\section{\bigskip Directional convexity and concluding corollary}

This section is motivated by the main result of section 3.

Let $C\subset\mathbb{R}^{k}$ be an open and convex set and let a vector
$\mathbf{h}\in\mathbb{R}^{k}$, $\mathbf{h}\neq\mathbf{0,}$ be fixed. Then, for
every $\mathbf{x}\in C$, the set
\[
I_{\mathbf{x},\mathbf{h}}:=\left\{  t\in\mathbb{R}:\mathbf{x}+t\mathbf{h}\in
C\right\}
\]
is a nonempty open interval and $0\in I_{\mathbf{x},\mathbf{h}}.$

We say that a function $f:C\rightarrow\mathbb{R}$ is \textit{convex (concave,
affine) in the direction of the vector} $\mathbf{h}$, if for every
$\mathbf{x}\in C,$ the function $\varphi_{\mathbf{x},\mathbf{h}}%
:I_{\mathbf{x},\mathbf{h}}\rightarrow\mathbb{R}$ defined by
\[
\varphi_{\mathbf{x},\mathbf{h}}\left(  t\right)  :=f\left(  \mathbf{x}%
+t\mathbf{h}\right)  \text{, \ \ \ \ \ \ }t\in I_{\mathbf{x},\mathbf{h}},
\]
is convex (concave, affine).

Of course, \ $f:C\rightarrow\mathbb{R}$ is \textit{convex in the direction of
the vector} $\mathbf{h},$ iff, for every $r\in\mathbb{R}$, $r\neq0$, the
function $f:C\rightarrow\mathbb{R}$ is \textit{convex in the direction of the
vector} $r\mathbf{h}.$

\begin{remark}
Let $C\subset\mathbb{R}^{k}$ be an open convex set and $\mathbf{h}%
\in\mathbb{R}^{k}$, and suppose that $f:C\rightarrow\mathbb{R}$ is twice
differentiable. Then

(i) $\ f$ is \textit{convex in the direction of the vector} $\mathbf{h}$, iff
for every $\mathbf{x}\in C,$
\[
\left(  \varphi_{\mathbf{x},\mathbf{h}}\right)  ^{\prime\prime}\left(
0\right)  \geq0;
\]

(ii) if $k=2$ and $\ \mathbf{h}=\left(  u,v\right)  ,$ $u^{2}+v^{2}\neq0;$
then $f$ is convex in the direction of the vector $\left(  u,v\right)  $, iff
for every $\left(  x,y\right)  \in C,$
\[
\left(  \varphi_{\left(  x,y\right)  ,\left(  u,v\right)  }\right)
^{\prime\prime}\left(  0\right)  =\frac{\partial^{2}f}{\partial x^{2}}\left(
x,y\right)  u^{2}+2\frac{\partial^{2}f}{\partial x\partial y}\left(
x,y\right)  uv+\frac{\partial^{2}f}{\partial x^{2}}\left(  x,y\right)
v^{2}\geq0;
\]
moreover, $f$ is convex, iff for every vector $\left(  u,v\right)
\in\mathbb{R}^{2}$, the function $f$ is convex in the direction of the vector
$\left(  u,v\right)  $, that is iff, (\cite{Hardy} pp. 78-81), for all
$\left(  x,y\right)  \in C$ and for all $\left(  u,v\right)  \in\mathbb{R}%
^{2},$%
\[
\left(  \varphi_{\left(  x,y\right)  ,\left(  u,v\right)  }\right)
^{\prime\prime}\left(  0\right)  =\frac{\partial^{2}f}{\partial x^{2}}\left(
x,y\right)  u^{2}+2\frac{\partial^{2}f}{\partial x\partial y}\left(
x,y\right)  uv+\frac{\partial^{2}f}{\partial x^{2}}\left(  x,y\right)
v^{2}\geq0.
\]

\end{remark}

In particular, in view of the second part of this remark, a twice
differentiable function $f$ defined in a convex set $C\subset\mathbb{R}^{2}$
is convex in the direction of the vector $\left(  1,1\right)  $ (or in the
direction of the main diagonal), iff%
\[
\frac{\partial^{2}f}{\partial x^{2}}\left(  x,y\right)  +2\frac{\partial^{2}%
f}{\partial x\partial y}\left(  x,y\right)  +\frac{\partial^{2}f}{\partial
x^{2}}\left(  x,y\right)  \geq0,\text{ \ \ \ \ \ }\left(  x,y\right)  \in C.
\]

\begin{example}
The function $f:\mathbb{R}^{2}\rightarrow\mathbb{R}$, $\ f\left(  x,y\right)
=x^{2}+5xy+y^{2}$, is not convex in the direction of the vector $\left(
1,-1\right)  $, as $\left(  \varphi_{\left(  x,y\right)  ,\left(  1,-1\right)
}\right)  ^{\prime\prime}\left(  0\right)  =-1$ for every $\left(  x,y\right)
\in\mathbb{R}^{2};$ in particular, $f$ it is not convex. Since $\left(
\varphi_{\left(  x,y\right)  ,\left(  1,1\right)  }\right)  ^{\prime\prime
}\left(  0\right)  =9$ for all $\left(  x,y\right)  \in\mathbb{R}^{2}$, the
function $f$ is convex in the direction of the vector $\left(  1,1\right)  .$
\end{example}

\bigskip Applying Theorem 4, we obtain the following

\begin{corollary}
A function $B:\left(  0,\infty\right)  ^{2}\rightarrow\left(  0,\infty\right)
$ is symmetric, logarithmically convex in the direction of the vector $(1,1),$
and
\[
B\left(  x+1,x+y+1\right)  =\frac{x\left(  x+y\right)  }{\left(
2x+y+1\right)  \left(  2x+y\right)  }B\left(  x,y\right)  \text{,
\ \ \ \ \ }x,y>0;
\]%
\[
B\left(  1,1+y\right)  =\mathcal{B}\left(  1,1+y\right)  ,\text{ \ \ \ \ }y>0;
\]
if, and only if, $B=\mathcal{B}$.
\end{corollary}

Let us mention that some notions of directional convexity appeared in real
analysis \cite{Matousek} as well as in complex analysis \cite{Hengart}.

\bigskip
\end{document}